\documentclass[a4paper,12pt,reqno]{amsart}
\usepackage{amsfonts,amsmath,amssymb,amstext,amsxtra}
\usepackage{mathabx}
\usepackage{euscript}
\usepackage{inputenc}
\usepackage[english]{babel}
\usepackage[T1]{fontenc}
\usepackage{enumitem}

\setlist[itemize]{leftmargin=*}

\newtheorem{theorem}{Theorem}[section]

\newtheorem{corollary}[theorem]{Corollary}
\newtheorem{remark}[theorem]{Remark}

\newcommand\C{\mathbb{C}}
\renewcommand\coth{\operatorname{coth}}
\renewcommand\epsilon{\varepsilon}

\newcommand\K{\mathcal{K}}
\newcommand\N{\mathbb{N}}
\renewcommand\phi{\varphi}

\renewcommand\Re{\operatorname{Re}}
\newcommand\R{\mathbb{R}}
\newcommand\sign{\operatorname{sign}}
\newcommand\ssb{\hspace{-.25mm}}
\newcommand\ssf{\hspace{.25mm}}
\renewcommand\tanh{\operatorname{tanh}}
\newcommand\V{\mathcal{V}}
\newcommand\Vt{\mathcal{V}^{\ssf t}}
\newcommand\vsb{\hspace{-.1mm}}
\newcommand\vsf{\hspace{.1mm}}

\title[Intertwining operator in 1D trigonometric Dunkl theory]
{An elementary proof of the positivity\\
of the intertwining operator\\
in one--dimensional trigonometric Dunkl theory}
\author[Anker]
{Jean--Philippe Anker}
\address{Universit\'e d'Orl\'eans \& CNRS,
F\'ed\'eration Denis Poisson (FR 2964),
Laboratoire MAPMO (UMR 7349),
B\^atiment de Math\'ematiques,
B.P.~6759, 45067 Orl\'eans cedex 2, France}
\email{anker@univ-orleans.fr}
\thanks{Work partially supported by the regional project MADACA
(Marches Al\'eatoires et processus de Dunkl\;--\;Approches
Combinatoires et Alg\'ebriques, www.fdpoisson.fr/madaca).}
\begin{document}

\begin{abstract}
This note is devoted to the intertwining operator
in the one--dimensional trigonometric Dunkl setting.
We obtain a simple integral expression of this operator
and deduce its positivity.
\end{abstract}

\maketitle
\vspace{-2mm}

\centerline{\it
To appear in Proc. Amer. Math. Soc.}
\bigskip

\section{Introduction}

We use the lecture notes \cite{Opdam2000}
as a general reference about trigonometric Dunkl theory.
In dimension 1, this special function theory is
a deformation of Fourier analysis on \ssf$\R\ssf$,
depending on two complex parameters \ssf$k_1$ and \ssf$k_2$\vsf,
where the classical derivative is replaced by the Cherednik operator
\begin{align*}
D\ssb f(x)
&=\bigl(\tfrac{d\vphantom{|}}{d\vsf x\vphantom{|}}\bigr)f(x)
+\bigl\{\tfrac{k_1\vphantom{|}}{1\ssf-\,e^{-x}}\ssb
+\ssb\tfrac{2\,k_2\vphantom{|}}{1\ssf-\,e^{-2\ssf x}}\bigr\}\,
\bigl\{\ssf f(x)\!-\!f(-x)\ssf\bigr\}
-\bigl(\tfrac{k_1\vphantom{|}}2\!+\hspace{-.5mm}k_2\bigr)\ssf f(x)\\
&=\bigl(\tfrac{d\vphantom{|}}{d\vsf x\vphantom{|}}\bigr)f(x)
+\bigl\{\ssf\tfrac{k_1\vsb+\ssf k_2\vphantom{|}}{2\vphantom{|}}\ssf
\coth\tfrac{x\vphantom{|}}{2\vphantom{|}}\ssb
+\ssb\tfrac{k_2\vphantom{|}}{2\vphantom{|}}\ssf
\tanh\tfrac{x\vphantom{|}}{2\vphantom{|}}\ssf\bigr\}\,
\bigl\{\ssf f(x)\!-\!f(-x)\ssf\bigr\}
-\bigl(\tfrac{k_1\vphantom{|}}2\!+\hspace{-.5mm}k_2\bigr)\ssf f(-x)\,,
\end{align*}
the Lebesgue measure by \ssf$A(x)\hspace{.5mm}dx$\ssf, where
\begin{equation*}
A(x)\vsb=|\ssf2\vsf\sinh\ssb\tfrac{x\vphantom{|}}{2\vphantom{|}}
\ssf|^{2\ssf k_1}\hspace{.5mm}
|\ssf2\vsf\sinh\vsb x\ssf|^{2\ssf k_2}\,,
\end{equation*}
and the exponential function \ssf$e^{\ssf i\ssf\lambda\ssf x}$
by the Opdam hypergeometric function
\begin{equation*}
\hspace{-15mm}G_{\vsf i\vsf\lambda}(x)
=\hspace{-20.25mm}\underbrace{
\phi_{\,2\ssf\lambda}^{\ssf k_1\ssb+\ssf k_2-\frac12,\hspace{.5mm}k_2-\frac12}
(\tfrac{x\vphantom{|}}{2\vphantom{|}})
}_{{}_2\vsf\text{\rm F}_{\ssb1\vsb}(
\frac{k_1}2\vsb+\vsf k_2+\ssf i\vsf\lambda\vsf,
\frac{k_1}2\vsb+\vsf k_2-\ssf i\vsf\lambda\hspace{.2mm};
\ssf k_1\ssb+\vsf k_2\vsb+\vsb\frac12\vsf;
\ssf-\sinh^2\hspace{-.5mm}\frac x2)\hspace{8mm}}\hspace{-20.5mm}
+\,\tfrac{\frac{k_1}2+\ssf k_2+\ssf i\ssf\lambda\vphantom{\frac0|}}
{2\ssf k_1\vsb+\ssf2\ssf k_2+1\vphantom{\frac|0}}\,(\sinh x)\,
\hspace{-21.75mm}\underbrace{
\phi_{\,2\ssf\lambda}^{\ssf k_1\ssb+\ssf k_2+\frac12,\hspace{.5mm}k_2+\frac12}
(\tfrac{x\vphantom{|}}{2\vphantom{|}})
}_{{}_2\vsf\text{\rm F}_{\ssb1\vsb}(
\frac{k_1}2\vsb+\vsf k_2+1\vsb+\ssf i\vsf\lambda\vsf,
\frac{k_1}2\vsb+\vsf k_2+1\vsb-\ssf i\vsf\lambda\hspace{.2mm};
\ssf k_1\ssb+\vsf k_2\vsb+\vsb\frac32\vsf;
\ssf-\sinh^2\hspace{-.5mm}\frac x2)\hspace{2mm}}\hspace{-20.75mm}.
\end{equation*}
Here \hspace{.5mm}$\phi_\lambda^{\alpha,\beta}(x)$
\ssf denotes the Jacobi function and
\ssf${}_2\ssf\text{\rm F}_{\ssb1}(a\ssf,\ssb b\ssf;\ssb c\ssf;\ssb Z\vsf)$
\ssf the classical hypergeometric function.
\smallskip

In a series of papers
(\cite{GallardoTrimeche2010}, \cite{Mourou2010},
\cite{Trimeche2010}, \cite{GallardoTrimeche2012},
\cite{Trimeche2012}, \cite{Trimeche2015a},
\cite{Trimeche2015b}, \cite{Trimeche2016}, \dots ),
Trim\`eche and his collaborators have studied an intertwining operator
\hspace{.5mm}$\V\hspace{-.5mm}:\hspace{-.5mm}
C^{\vsf\infty\vsb}(\R)\ssb\longrightarrow\ssb C^{\vsf\infty\vsb}(\R)$\ssf,
which is characterized by
\begin{equation*}
\V\ssf\circ\bigl(\tfrac{d\vphantom{|}}{d\vsf x\vphantom{|}}\bigr)\vsb
=D\circ\ssf\V
\quad\text{and}\quad
\delta_{\vsf0}\ssb\circ\V=\delta_{\vsf0}\,,
\end{equation*}
and the dual operator \hspace{.5mm}$\Vt\!:\hspace{-.5mm}
C_c^{\vsf\infty\vsb}(\R)\ssb\longrightarrow\ssb C_c^{\vsf\infty\vsb}(\R)$\ssf,
which satisfies
\begin{equation*}
\int_{-\infty}^{\ssf+\infty}\hspace{-1mm}\V\ssb f(x)\,g(x)\,A(x)\,dx\ssf
=\int_{-\infty}^{\ssf+\infty}\hspace{-1mm}
f(y)\,\Vt\hspace{-.4mm}g(y)\,dy\,.
\end{equation*}
Let us mention in particular the following facts.
\begin{itemize}

\item
\textit{Eigenfunctions\/}.
For every \ssf$\lambda\hspace{-.6mm}\in\hspace{-.5mm}\C$\ssf,
\begin{equation*}
\V\ssf(\ssf x\ssb\longmapsto\ssb e^{\ssf i\ssf\lambda\ssf x}\ssf)\vsb
=G_{\vsf i\vsf\lambda}\,.
\end{equation*}

\item
\textit{Explicit expression\/}.
An integral representation of \hspace{.5mm}$\V$ \ssf was computed
in \cite{GallardoTrimeche2010} (and independently in \cite{Ayadi2011}),
under the assumption that \ssf$k_1\hspace{-.8mm}\ge\hspace{-.5mm}0$\ssf,
\ssf$k_2\!\ge\hspace{-.5mm}0$ \ssf with
\ssf$k_1\hspace{-.8mm}+\hspace{-.5mm}k_2\!>\hspace{-.5mm}0$\ssf.

\item
\textit{Analytic continuation\/}.
It was shown in \cite{GallardoTrimeche2012} that
the intertwining operator \ssf$\V$ extends meromorphically
with respect to \ssf$k\hspace{-.5mm}\in\hspace{-.5mm}\C^2$,
with singularities in
\ssf$\{\hspace{.4mm}k\hspace{-.5mm}\in\hspace{-.5mm}\C^2\hspace{.5mm}|\,
k_1\hspace{-.9mm}+\hspace{-.5mm}k_2\!+\!\frac12\!\in\!-\ssf\N\ssf\}$\ssf.

\item
\textit{Positivity\/}.
On the one hand, the positivity of \hspace{.5mm}$\V$
\ssf was disproved in \cite{GallardoTrimeche2010},
by using the above--mentioned expression of \hspace{.5mm}$\V$ \ssf in the case
\ssf$k_1\hspace{-.8mm}>\hspace{-.5mm}0$\ssf, $k_2\!>\hspace{-.5mm}0$\ssf.
On the other hand, the positivity of \hspace{.5mm}$\V$
\ssf was investigated in \cite{Trimeche2012}, \cite{Trimeche2015a},
\cite{Trimeche2015b}, \cite{Trimeche2016}
by using the positivity of a heat type kernel in the case
\ssf$k_1\hspace{-.8mm}\ge\hspace{-.5mm}0$\ssf, $k_2\!\ge\hspace{-.5mm}0$\ssf.

\end{itemize}

In Section 2, we obtain an integral representation of \hspace{.5mm}$\V$ \ssf
and \hspace{.5mm}$\Vt$ when \ssf$\Re k_1\hspace{-.8mm}>\hspace{-.5mm}0$
\ssf and \ssf$\Re k_2\!>\hspace{-.5mm}0\ssf$.
The expression is simpler and the proof is quicker than the previous ones in
\cite{GallardoTrimeche2010} or \cite{Ayadi2011}.
In Section 3, we deduce the positivity
of \hspace{.5mm}$\V$ \ssf and \hspace{.5mm}$\Vt$
when \ssf$k_1\hspace{-.8mm}>\hspace{-.5mm}0$\ssf, $k_2\!>\hspace{-.5mm}0$\ssf,
and comment on the positivity issue.

\section{Integral representation of the intertwining operator}

In this section,
we resume the computations in \cite[\!Section\;2\vsf]{GallardoTrimeche2010}
%(see also \cite[\!Ch.\;4\vsf]{Ayadi2011}\vsf)
and prove the following result.

\begin{theorem}\label{MainTheorem}
Let \,$k\hspace{-.5mm}=\hspace{-.5mm}(k_1,\ssb k_2)
\hspace{-.5mm}\in\hspace{-.4mm}\C^2$
with \,$\Re k_1\hspace{-.8mm}>\hspace{-.5mm}0$
and \,$\Re k_2\!>\hspace{-.5mm}0$\ssf.
Then
\vspace{-1mm}
\begin{equation*}
\V f(x)=\int_{\ssf|y|<|x|}\!\K(x,y)\,f(y)\,dy
\qquad\forall\;x\hspace{-.5mm}\in\hspace{-.5mm}\R^*
\end{equation*}
\vspace{-3mm}

and
\vspace{-1mm}
\begin{equation*}
\Vt\ssb g\vsf(y)=\int_{\ssf|x|>|y|}\!\K(x,y)\,g(x)\,A(x)\,dx\,,
\end{equation*}
\vspace{-3mm}

where
\begin{equation}\begin{aligned}\label{IntegralExpression}
\K(x,y)=\tfrac{c\vphantom{|}}{4\vphantom{|}}\hspace{.5mm}A(x)^{-1}\ssf
&\smash{\int_{\,|y|}^{\,|x|}}
\sigma(x\vsf,\vsb y\vsf,\ssb z)\,
(\cosh\tfrac z2\hspace{-.5mm}-\ssb\cosh\tfrac y2\vsf)^{\vsf k_1\vsb-1}
\hspace{.5mm}(\cosh x\ssb-\ssb\cosh z\vsf)^{\vsf k_2-1}
\vphantom{\Big|}\\
&\hspace{6mm}\times
(\vsf\sinh\tfrac z2\vsf)\,dz\,,
\vphantom{\Big|}
\end{aligned}\end{equation}
\vspace{-1.5mm}

with
\vspace{-.5mm}
\begin{equation}\label{constantc}
c\ssf=\ssf2^{\ssf3\ssf k_1\vsb+\ssf3\ssf k_2}\hspace{.5mm}
\smash{\tfrac{\Gamma(k_1\vsb+\ssf k_2+\frac12)\vphantom{\frac0|}}
{\sqrt{\pi\ssf}\,\Gamma(k_1)\,\Gamma(k_2)\vphantom{\frac|0}}}
\end{equation}
and
\begin{equation}\label{sigmaxyz}
\sigma(x\vsf,\vsb y\vsf,\ssb z)=(\vsf\sign x)\,
\bigl\{\ssf e^{\ssf\frac x2}\hspace{.5mm}(\vsf2\cosh\tfrac x2\vsf)\ssb
-e^{-\frac y2}\hspace{.5mm}(\vsf2\cosh\tfrac z2\vsf)\bigr\}\,.
\end{equation}
\end{theorem}

\begin{proof}
%Let \ssf$\rho\ssb=\hspace{-.5mm}\frac{k_1}2\!+\hspace{-.5mm}k_2$\ssf.
As observed in \cite{GallardoTrimeche2010} and \cite{Mourou2010},
\vspace{-1mm}
\begin{equation*}
\V f(x)=\int_{-|x|}^{\vsf+|x|}\!\K(x,y)\,f(y)\,dy
\end{equation*}
is an integral operator, whose kernel
\begin{equation}\label{MourouFormula}
\K(x,y)=\tfrac14\,K(\tfrac x2,\tfrac y2)
+(\vsf\sign x)\hspace{.5mm}\bigl(\tfrac{k_1\vphantom{|}}4\hspace{-.5mm}
+\hspace{-.5mm}\tfrac{k_2\vphantom{|}}2\bigr)\hspace{.5mm}A(x)^{-1}\ssf
\widetilde{K}(\tfrac x2,\tfrac y2)\vsb
-(\vsf\sign x)\,\tfrac12\,A(x)^{-1}\hspace{.5mm}
\tfrac{\partial}{\partial y}\,\widetilde{K}(\tfrac x2,\tfrac y2)
\end{equation}
\vspace{-4mm}

can be expressed in terms of the kernel
\begin{equation}\label{ExpressionK}\begin{aligned}
K(x,y)
&=\ssf2\,c\,A(2\ssf x)^{-1}\hspace{.5mm}|\ssb\sinh2\ssf x\ssf|
\vphantom{\Big|}\\
&\ssf\times\smash{\int_{\,|y|}^{\,|x|}}\ssb
(\cosh z\ssb-\ssb\cosh y\vsf)^{\vsf k_1-1}\hspace{.5mm}
(\cosh2\ssf x\ssb-\ssb\cosh2\ssf z\vsf)^{\vsf k_2-1}\hspace{.5mm}
(\vsf\sinh z)\,dz\vphantom{\Big|}
\end{aligned}\end{equation}
\vspace{0mm}

of the intertwining operator in the Jacobi setting
(see \cite[\!Subsection\;5.3\vsf]{Koornwinder1984}\vsf)
and of its
\linebreak
integral
\begin{equation}\label{Expression1Ktilde}\begin{aligned}
\widetilde{K}(x,y)
&=\smash{\int_{\,|y|}^{\ssf|x|}}\!K(w,y)\,A(2\ssf w)\,dw\vphantom{\Big|_|}\\
&=\tfrac{c\vphantom{|}}{k_2\vphantom{|}}\,\smash{\int_{\,|y|}^{\,|x|}}\ssb
(\cosh z\ssb-\ssb\cosh y\vsf)^{\vsf k_1-1}\hspace{.5mm}
(\cosh2\ssf x\ssb-\ssb\cosh2\ssf z\vsf)^{\vsf k_2}\hspace{.5mm}
(\vsf\sinh z)\,dz\,.\vphantom{\Big|^|}
\end{aligned}\end{equation}
\vspace{0mm}

Let us integrate by parts \eqref{Expression1Ktilde}
and differentiate the resulting expression with respect to \ssf$y\vsf$.
This way, we obtain
\vspace{1mm}
\begin{equation}\label{Expression2Ktilde}
\widetilde{K}(x,y)
=\tfrac{4\hspace{.5mm}c\vphantom{|}}{k_1\vphantom{|}}
\hspace{.5mm}\smash{\int_{\,|y|}^{\,|x|}}\ssb
(\cosh z\ssb-\ssb\cosh y\vsf)^{\vsf k_1}\hspace{.5mm}
(\cosh2\ssf x\ssb-\ssb\cosh2\ssf z\vsf)^{\vsf k_2-1}\hspace{.5mm}
(\cosh z\vsf)\,(\vsf\sinh z\vsf)\,dz\,,\vphantom{\Big|^|}
\end{equation}
\vspace{-2.5mm}

and
\vspace{-.5mm}
\begin{equation}\begin{aligned}\label{DerivativeKtilde}
\tfrac{\partial}{\partial y}\hspace{.5mm}\widetilde{K}(x,y)
=-\,4\,c\,(\vsf\sinh y\vsf)\,\smash{\int_{\,|y|}^{\,|x|}}\ssb
&(\cosh z\ssb-\ssb\cosh y\vsf)^{\vsf k_1\vsb-1}\hspace{.5mm}
(\cosh2\ssf x\ssb-\ssb\cosh2\ssf z\vsf)^{\vsf k_2-1}\vphantom{\Big|^|}\\
\times\;&(\cosh z\vsf)\,(\vsf\sinh z\vsf)\,dz\,.\vphantom{\Big|^|}
\end{aligned}\end{equation}
We conclude by substituting
\eqref{ExpressionK}, \eqref{Expression1Ktilde},
\eqref{Expression2Ktilde}, \eqref{DerivativeKtilde}
in \eqref{MourouFormula}
and more precisely \eqref{Expression1Ktilde},
respectively \eqref{Expression2Ktilde} in
\vspace{-2mm}
\begin{equation*}
(\vsf\sign x)\,\tfrac{k_2\vphantom{|}}{2\vphantom{|}}\hspace{.5mm}
A(x)^{-1}\ssf\widetilde{K}(\tfrac x2,\tfrac y2)\,,
\quad\text{respectively}\quad
(\vsf\sign x)\,\tfrac{k_1\vphantom{|}}{4\vphantom{|}}\hspace{.5mm}
A(x)^{-1}\ssf\widetilde{K}(\tfrac x2,\tfrac y2)\,.
\end{equation*}
\vspace{-7mm}

\end{proof}

\begin{remark}
Let \,$x,y\!\in\!\R$ such that \,$|x|\!>\!|y|$\ssf.
The expression \eqref{IntegralExpression} extends meromorphically
with respect to \ssf$k\hspace{-.5mm}\in\hspace{-.5mm}\C^2$,
with singularities in
\,$\{\hspace{.4mm}k\hspace{-.5mm}\in\hspace{-.5mm}\C^2\hspace{.5mm}|\,
k_1\hspace{-.9mm}+\hspace{-.5mm}k_2\!+\!\frac12\!\in\!-\ssf\N\ssf\}$\ssf,
produced by the factor
\,$\Gamma(k_1\hspace{-1mm}+\hspace{-.5mm}k_2\!+\!\frac12)$
in \eqref{constantc}.
In the limit cases where either \,$k_1$ \ssb or \,$k_2$ vanishes,
\eqref{IntegralExpression} reduces to the following expressions,
already obtained in \ssf\cite{GallardoTrimeche2010}
\ssb and \ssf\cite{Ayadi2011}\,$:$
\begin{itemize}

\item
Assume that \,$k_1\hspace{-1mm}=\hspace{-.5mm}0$
and \,$\Re k_2\!>\hspace{-.5mm}0$\ssf. Then
\begin{equation}\label{IntegralExpressionLimit1}
\K(x,y)=2^{\ssf k_2-1}\hspace{.5mm}
\tfrac{\Gamma(k_2+\frac12)\vphantom{\frac0|}}
{\sqrt{\pi\ssf}\,\Gamma(k_2)\vphantom{\frac|0}}\,
|\ssb\sinh x\ssf|^{-2\ssf k_2}\hspace{.5mm}
(\cosh x\ssb-\ssb\cosh y\vsf)^{\vsf k_2-1}\hspace{.5mm}
(\vsf\sign x)\,(e^{\ssf x}\!-\ssb e^{-\ssf y}\ssf)\,.
\end{equation}

\item
Assume that \,$k_2\!=\hspace{-.5mm}0$
and \,$\Re k_1\hspace{-1mm}>\hspace{-.5mm}0$\ssf.
Then
\begin{equation}\label{IntegralExpressionLimit2}
\K(x,y)=2^{\ssf k_1\vsb-2}\hspace{.5mm}
\tfrac{\Gamma(k_1\vsb+\frac12)\vphantom{\frac0|}}
{\sqrt{\pi\ssf}\,\Gamma(k_1)\vphantom{\frac|0}}\,
|\ssb\sinh\tfrac x2\ssf|^{-2\ssf k_1}\hspace{.5mm}
(\cosh\tfrac x2\hspace{-.5mm}-\ssb\cosh\tfrac y2)^{\vsf k_1\vsb-1}\hspace{.5mm}
(\vsf\sign x)\,(e^{\ssf\frac x2}\!-\ssb e^{-\ssf\frac y2})\,.
\end{equation}

\end{itemize}
\end{remark}

\section{Positivity of the intertwining operator}

\begin{corollary}
Assume that \,$k_1\hspace{-1mm}>\!0$ and \,$k_2\!>\!0$\ssf.
Then the kernel \eqref{IntegralExpression} is strictly positive,
for every \,$x,\vsb y\hspace{-.5mm}\in\hspace{-.5mm}\R$
such that \,$|x|\!>\!|y|$\ssf.
Hence the intertwining operator \,$\V$ \ssf and its dual \,$\Vt$ are positive.
\end{corollary}

\begin{proof}
Let us check the positivity of \eqref{sigmaxyz}
when \ssf$x,\vsb y,\vsb z\hspace{-.5mm}\in\hspace{-.5mm}\R$
\ssf satisfy \ssf$|x|\!>\hspace{-.5mm}z\hspace{-.5mm}>\!|y|$\ssf.
On the one hand, if \ssf$x\hspace{-.5mm}>\hspace{-.5mm}0$\ssf, then
\begin{align*}
\sigma(x\vsf,\vsb y\vsf,\ssb z)
&=\vsf e^{\ssf\frac x2}\hspace{.5mm}(\vsf2\cosh\tfrac x2\vsf)\vsb
-e^{-\frac y2}\hspace{.5mm}(\vsf2\cosh\tfrac z2\vsf)\\
&>(\vsf e^{\ssf\frac x2}\hspace{-.5mm}-\ssb e^{-\frac y2})\,
(\vsf2\cosh\tfrac x2\vsf)>\vsf0\,.
\end{align*}
On the other hand, if \ssf$x\hspace{-.5mm}<\hspace{-.5mm}0$\ssf, then
\begin{align*}
\sigma(x\vsf,\vsb y\vsf,\ssb z)
&=\vsf e^{-\frac y2}\hspace{.5mm}(\vsf2\cosh\tfrac z2\vsf)\ssb
-e^{\ssf\frac x2}\hspace{.5mm}(\vsf2\cosh\tfrac x2\vsf)\\
&>\vsf e^{-\frac y2}\hspace{.5mm}(\vsf2\cosh\tfrac y2\vsf)\ssb
-e^{\ssf\frac x2}\hspace{.5mm}(\vsf2\cosh\tfrac x2\vsf)
=\vsf e^{-y}\ssb-e^{\ssf x}>\vsf0\,.
\end{align*}
\vspace{-7mm}

\end{proof}

\begin{remark}
As already observed in \ssf\cite{GallardoTrimeche2010}\vsf,
the positivity of \eqref{IntegralExpressionLimit1},
respectively \eqref{IntegralExpressionLimit2}
is im\-me\-di\-ate in the limit case where
\,$k_1\hspace{-1mm}=\hspace{-.5mm}0$
and \,$k_2\!>\hspace{-.5mm}0$\ssf,
respectively \,$k_2\!=\hspace{-.5mm}0$
and \,$k_1\hspace{-1mm}>\hspace{-.5mm}0\ssf$.
\end{remark}

\begin{remark}
The positivity of \,$\V$ was mistakenly disproved in
\ssf\cite[\!Theorem\;2.11]{GallardoTrimeche2010} \ssb
when \,$k_1\hspace{-1mm}>\hspace{-.5mm}0$ and \,$k_2\!>\hspace{-.5mm}0$\ssf.
More precisely,
by using a more complicated formula than \eqref{IntegralExpression},
the density \,$\K(x,y)$ was shown to be negative,
when \hspace{.5mm}$x\hspace{-.5mm}>\hspace{-.5mm}0$
and \hspace{.5mm}$y\!\searrow\!-\ssf x$\ssf.
The error in the proof lies in the expression \hspace{.5mm}$A_1$,
which is equal to
\hspace{.5mm}$\tfrac{k\vphantom{|}}{k^{\ssf\prime}\vphantom{|}}
\hspace{.5mm}\tfrac{\sinh(2\ssf x)-\ssf\sinh(2|y|)}{E\vphantom{|}}$
\ssf and which tends to
$+\hspace{.5mm}2\hspace{.5mm}
\smash{\tfrac{k\vphantom{|}}{k^{\ssf\prime}\vphantom{|}}}\hspace{.5mm}
\smash{\tfrac{\cosh(2\ssf x)\vphantom{|}}{\sinh(2\ssf x)\vphantom{|}}}
\ssb>\ssb0\ssf$.
\end{remark}

\begin{remark}
A different approach, based on the positivity of a heat type kernel,
was used in \ssf\cite{Trimeche2012}\vsf, \cite{Trimeche2015a}\vsf,
\cite{Trimeche2015b} \ssb and \ssf\cite{Trimeche2016}
in order to tackle the positivity of \,$\V$.
While \ssf\cite{Trimeche2012} \ssb may be right,
the same flaw occurs in \ssf\cite{Trimeche2015a}\vsf,
\cite{Trimeche2015b}\vsf, \cite{Trimeche2016}\vsf,
namely the cut--off \,$\mathbf{1}_{\vsf Y_\ell}$
\ssb breaks down the differential--difference equations,
which are not local.
\end{remark}

In conclusion, the present note settles in a simple way the positivity issue in dimension 1
and hence in the product case. Otherwise,
the positivity of the interwining operator \ssf$\V$ and its dual \ssf$\Vt$,
when the multiplicity function \ssf$k$ \ssf is \ssf$\ge\hspace{-.5mm}0$\ssf,
remains an open problem in higher dimensions.


\begin{thebibliography}{99}

%\bibitem{DLMF}
%\textit{DLMF\/} (\textit{Digital Library of Mathematical Functions\/}),
%NIST (National Institute of Standards and Tech\-no\-logy),
%http\ssf:/\!/dlmf.nist.gov/

%\bibitem{Anker2016}
%J.--Ph. Anker,\,:
%\textit{An introduction to Dunkl theory and its analytic aspects\/},
%to appear in
%\textit{Analytic, Algebraic and Geometric Aspects of Differential Equations\/}
%(\textit{\foreignlanguage{polish}{B"edlewo}, Poland, September 2015\/}),
%G.~Filipuk, Y.~Haraoka \& S.~Michalik (eds.),
%Trends Math., Birkh\"auser.

\bibitem{Ayadi2011}
F. Ayadi\,:
\textit{Analyse harmonique et \'equation de Schr\"odinger
associ\'ees au laplacien de Dunkl trigonom\'etrique\/},
Ph.D. Thesis, Universit\'e d'Orl\'eans \& Universit\'e de Tunis El Manar (2011),
https\,:\ssf/\!/tel.archives-ouvertes.fr/tel--00664822.

\bibitem{GallardoTrimeche2010}
L. Gallardo \& K. Trim\`eche\,:
\textit{Positivity of the Jacobi--Cherednik intertwining operator and its dual\/},
Adv. Pure Appl. Math. 1 (2010), no. 2, 163--194.

\bibitem{GallardoTrimeche2012}
L. Gallardo \& K. Trim\`eche\,:
\textit{Singularities and analytic continuation
of the Dunkl and the Jacobi--Cherednik intertwining operators and their duals\/},
J. Math. Anal. Appl. 396 (2012), no.1, 70--83.

\bibitem{Koornwinder1984}
T.H. Koornwinder\,:
\textit{Jacobi functions and analysis on noncompact semisimple Lie groups\/},
in \textit{Special functions\/} (\textit{group theoretical aspects and applications\/}),
R.A. Askey, T.H. Koornwinder \& W. Schempp (eds.), pp.\;1--84,
Reidel, Dordrecht, 1984.

\bibitem{Mourou2010}
M. A. Mourou\,:
\textit{Transmutation operators and Paley-Wiener theorem
associated with a Cherednik type operator on the real line\/},
Anal. Appl. 8 (2010), no. 4, 387--408.

%\bibitem{Opdam1995}
%E.M. Opdam\,:
%\textit{Harmonic analysis for certain representations of graded Hecke algebras\/},
%Acta Math. 175 (1995), 75--121.

\bibitem{Opdam2000}
E.M. Opdam\,:
\textit{Lecture notes on Dunkl operators for real and complex reflection groups\/},
Math. Soc. Japan Mem. 8, Math. Soc. Japan, Tokyo, 2000.

%\bibitem{Roesler2003}
%M. R{\"o}sler\,:
%\textit{Dunkl operators\/} (\textit{theory and applications\/}),
%in \textit{Orthogonal polynomials and special functions\/} (\textit{Leuven, 2002\/}),
%E. Koelink \& W. Van Assche (eds.), pp.\;93--135,
%Lect. Notes Math. 1817, Springer--Verlag, Berlin, 2003.

\bibitem{Trimeche2010}
K. Trim\`eche\,:
\textit{The trigonometric Dunkl intertwining operator and its dual
associated with the Cherednik operators and the Heckman--Opdam theory\/},
Adv. Pure Appl. Math. 1 (2010), no. 3, 293--323.

\bibitem{Trimeche2012}
K. Trim\`eche\,:
\textit{Positivity of the transmutation operators
associated with a Cherednik type operator on the real line\/},
Adv. Pure Appl. Math. 3 (2012), 361--376.

\bibitem{Trimeche2015a}
K. Trim\`eche\,:
\textit{Positivity of the transmutation operators
and absolute continuity of their representing measures
for a root system on \ssf$\R^d$},
Int. J. Appl. Math. 28 (2015), no. 4, 427--453.

\bibitem{Trimeche2015b}
K. Trim\`eche\,:
\textit{The positivity of the transmutation operators
associated with the Cherednik operators
attached to the root system of type \ssf$A_2$},
Adv. Pure Appl. Math. 6 (2015), no. 2, 125--134.

\bibitem{Trimeche2016}
K. Trim\`eche\,:
\textit{The positivity of the transmutation operators
associated to the Cherednik operators
for the root system \ssf$BC_2$\/},
Math. J. Okayama Univ. 58 (2016), 183--198.

\end{thebibliography}
\end{document}